\documentclass[reqno]{amsart}
\usepackage{amssymb,amsmath,times}
\usepackage[mathscr]{eucal}
\usepackage{xspace}
\usepackage{stix}
\usepackage[colorlinks = false]{hyperref}

\newcommand{\R}{\mathbb{R}}

\newcommand{\C}{\mathbb{C}}
\newcommand{\LL}{\mathcal{L}}

\newcommand{\po}{\partial}

\newcommand{\ve}{\varepsilon}

\newcommand{\X}{\times}

\renewcommand{\d}{\delta}
\renewcommand{\l}{\lambda}

\renewcommand{\a}{\alpha}

\newcommand{\s}{\sigma}
 
\newcommand{\z}{\zeta}
\renewcommand{\k}{\kappa}

\renewcommand{\div}{\text{\rm div}\,}

\newcommand{\cR}{{\mathcal R}}

\newcommand{\fkU}{{\frak U}}
\newcommand{\fkB}{{\frak B}}

\newcommand{\uu}{{\frak u}}

\newcommand{\fka}{{\frak a}}
\newcommand{\fkb}{{\frak b}}

\newcommand{\rhore}{{\rho_{(re)}}}
\newcommand{\ure}{{u_{(re)}}}
\newcommand{\vre}{{\vartheta_{(re)}}}

\theoremstyle{plain}
\newtheorem{theorem}{Theorem}[section]

\theoremstyle{definition}

\theoremstyle{remark}

\newtheorem{remark}{Remark}[section]
\numberwithin{equation}{section}

\begin{document}

%Topmatter
\title[Relativistic short wave-long wave interactions]
{On short wave-long wave interactions\\ in the relativistic context: \\ Application to the Relativistic Euler Equations.}

\author{Jo\~ao Paulo Dias}
\thanks{J.P.\ Dias gratefully acknowledges the support from FCT (Funda\c c\~ao para a Ci\^encia  e a Tecnologia)
under the project UIDB/04561/2020}
\address{Departamento de Matem\'atica and CMAFcIO \\
Faculdade de Ci\^encias, Universidade de Lisboa\\
Campo Grande, Edif.\ C6, 1749-016, Lisboa, Portugal }
\email{jpdias@fc.ul.pt}

\author{Hermano Frid} 
\thanks{H.~Frid gratefully acknowledges the support from CNPq, 
through grant proc.\ 305097/2019-9, and FAPERJ, 
through grant proc.\ E-26/202.900-2017}

\address{Department of  Computation and Mathematics-DCM\\
FFCLRP--USP, Av. Bandeirantes, 3900 - Monte Alegre
Ribeir\~ao Preto - SP - CEP.\ 14040-901}
\email{hermano.frid@usp.br}

\date{}



\begin{abstract}  In this paper we introduce a model of relativistic short wave-long 
wave interaction where the short waves are described by the massless $1+3$-dimensional Thirring model of nonlinear Dirac equation and  the long waves are described by the $1+3$-dimensional  relativistic Euler equations. The interaction coupling terms are modeled by a potential
proportional to the relativistic specific volume in the Dirac equation and an external force proportional  to the square modulus of the Dirac wave function in the relativistic Euler equation.  An important feature of the model is that the Dirac equations are based on the Lagrangian coordinates of the relativistic fluid flow. In particular, an important contribution of this paper is a clear formulation of  the relativistic Lagrangian transformation. This is done by means of the introduction of natural auxiliary dependent variables, rendering the discussion totally similar to the non-relativistic case. As far as the authors know the definition of the Lagrangian transformation given in this paper  is new. Finally, we establish the short-time existence  and uniqueness of a smooth solution of the Cauchy problem for the regularized model. This follows through the symmetrization of the relativistic Euler equation introduced by Makino and Ukai (1995) and requires a slight extension of a well known theorem of T.~Kato (1975) on quasi-linear symmetric hyperbolic systems.

\end{abstract}

\maketitle 

\section{Introduction} \label{S:1}

We consider the short wave-long wave interaction for a relativistic fluid described by the relativistic Euler equations in $\R^3$. The latter, when no external forces are acting, is given by (see, e.g.,  \cite{MU1, MU2, PS, LFU})
\begin{equation}\label{e1.1n}
\begin{aligned}
&\po_t\left(\frac{\rho+\ve^2 p}{1-\ve^2|u|^2}-\ve^2 p\right)+\sum_{k=1}^3\po_{x_k}\left(\frac{\rho+\ve^2p}{1-\ve^2|u|^2}u_k\right)=0,\\
&\po_t\left(\frac{\rho+\ve^2p}{1-\ve^2|u|^2}u_j\right)+\sum_{k=1}^3\po_{x_k}\left(\frac{\rho+\ve^2p}{1-\ve^2|u|^2}u_j u_k +p\d_{jk}\right)=0,\qquad j=1,2,3.
\end{aligned}
\end{equation}
Here, $\rho$ is the mass density and $u=(u_1,u_2,u_3)$ is the velocity vector of the fluid, both 
functions of $(t,x)\in\R_+\X\R^3$. The parameter $\ve$  represents the {\em inverse of the light speed}, $\d_{jk}$ is the Kronecker symbol, so that the $d\X d$ identity matrix is $I_d:=(\d_{jk})_{j,k=1}^d$, and $p=p(\rho)$ is the pressure. The physical domain for the unknown $(\rho,u)$ is 
$$
\rho\ge0,\qquad |u|^2:=\sum_{j=1}^3 u_j^2<\ve^{-2},
$$
while the pressure $p=p(\rho)$ satisfies
$0\le p'(\rho)<\ve^{-2}$.

In the relativistic context, the short waves  are described by the Dirac equation proposed by Dirac (\cite{Di}) in search of compatibility between relativity and quantum theories.  On the other hand as a replacement for the nonlinear cubic Schr\"odiger equation there are different models of 
the nonlinear cubic Dirac equation (see, e.g., \cite{CG,Ba, CH,Th, Del, KM, Ca}).  
Here, as in \cite{DF},  we will be concerned with the Thirring model proposed by Thirring in \cite{Th} whose mathematical study has been considered in several papers (see, e.g., \cite{Del, KM, DF1, Ca}). 
More specifically, here we only consider the zero mass case. 

In \cite{DF} two examples of models of short wave-long wave interactions in the relativistic context were addressed. Namely,   the case of a one-dimensional  scalar conservation law in the relativistic context such as the one introduced LeFloch et al.\ in \cite{LFMO} (see also \cite{HW}), and the interaction with long waves described by the augmented Born-Infeld (ABI) 
equations in electromagnetism, introduced by Brenier in \cite{Br}. Both of these examples were in one spatial variable.

In this paper we are concerned first with establishing a model of relativistic short wave-long wave interaction where the short waves are described by a massless  $1+3$-dimensional extension  of the Thirring model of nonlinear Dirac equation, with an interaction term representing the potential of an external force. On the other hand, the long waves are described by the $1+3$-dimensional
relativistic Euler equations \eqref{e1.1n}, with the interaction term appearing as an external force on the right-hand side of the relativistic momentum equations. 
An important feature of the model is that the Dirac equations are based on the Lagrangian coordinates of the relativistic fluid flow. In particular, an important contribution of this paper is a clear formulation of  the relativistic Lagrangian transformation. This is done by means of the introduction of natural auxiliary dependent variables, rendering the discussion totally similar to the non-relativistic case. As far as the authors know the definition of the relativistic Lagrangian transformation given in this paper  is new.
 We refer to \eqref{e1.123}, or its regularized form, \eqref{e1.123r}, for a view of the final model being proposed here. Then, we establish the local  in time existence and uniqueness of a smooth solution for the proposed regularized model \eqref{e1.123r}.
 This follows through the symmetrization introduced by Makino and Ukai  in \cite{MU2} and  requires a slight extension of the well known theorem of Kato in \cite{K} on quasi-linear symmetric hyperbolic systems.

 The rest of this paper is organized as follows. In Section~\ref{S:2} we discuss the Lagrangian transformation of the relativistic fluid flow. In Section~\ref{S:3} we recall the main property of the massless Thirring model obtained in \cite{DF1}. In Section~\ref{S:4} we establish our model for the relativistic interaction of short waves described by $1+3$-dimensional Thirring model of nonlinear  Dirac equation and long waves governed by the $1+3$-dimensional relativistic Euler equations, and its regularized version. In Section~\ref{S:5} we establish the local in time existence and uniqueness of a  smooth solution of the Cauchy problem for the regularized model.

\section{The Lagrangian transformation for the relativistic Euler equations}\label{S:2}

We are going to obtain first the Lagrangian transformation for \eqref{e1.1n}. For that it will be useful to introduce  the following auxiliary variables
\begin{equation}\label{e1.2n}
\begin{cases}
\rhore:=\dfrac{\rho+\ve^4|u|^2 p}{1-\ve^2|u|^2}, \quad &\text{(relativistic density)}\\
\ure:=\dfrac{\rho+\ve^2 p}{\rho+\ve^4|u|^2 p}u,\quad  &\text{(relativistic velocity)}\\
\tilde P_{jk}:=\dfrac{\rhore\ve^2(\ve^2|u|^2-1)}{\rho+\ve^2 p}\ure_j\ure_k, \quad &\text{(relativistic pressure loss tensor)}\\
& j,k=1,2,3.
\end{cases}
\end{equation}
Observing that $\rhore\ure= \frac{\rho+\ve^2p}{1-\ve^2|u|^2}u$, we may write \eqref{e1.1n} in the following illuminating form
\begin{equation}\label{e1.3n}
\begin{aligned}
&\po_t\rhore+\div (\rhore\ure)=0,\\
&\po_t(\rhore\ure)+\div(\rhore \ure\otimes\ure+ \tilde P+pI_3)=
\rhore F,
\end{aligned}
\end{equation}
but here we also allow for the action of an external force $F$. 
As far as the authors know, the above form of the relativistic Euler equations is presented in this paper for the first time, and it is very useful, in particular, in obtaining the Lagrangian transformation for the relativistic Euler equations.  Concerning the latter, due to the very similar structure of \eqref{e1.3n} with the non-relativistic Euler equations, we basically repeat the same procedures as for the non-relativistic case (see, e.g., \cite{DF, FPZ}). So, let us consider the flow $\Phi(t;x)$ given by the differential equation
\begin{equation}\label{e1.4n}
\begin{cases}
\dfrac{d \Phi}{dt}(t;x)&=\ure(t,\Phi(t;x))\\
\Phi(0;x)&=x.
\end{cases}
\end{equation}
As in the non-relativistic case, we consider the Jacobian matrix 
$$
J_{\Phi}(t;x):=\det\left(\dfrac{\po \Phi}{\po x}(t;x)\right),
$$
for which we have the well known Euler identity for the Jacobian, (see, e.g., \cite{Bat}):  
\begin{equation}\label{e1.5n}
\begin{cases}
\dfrac{d J_{\Phi}(t;x)}{dt}=\div \ure(t,\Phi(t;x))\,J_{\Phi}(t;x),\\
J_{\Phi}(0;x)=1.
\end{cases}
\end{equation}
The {\em Lagrangian transformation} 
\begin{equation}\label{e1.6n0}
Y(t,x)=(t,y(t,x))
\end{equation}  
is defined by
\begin{equation}\label{e1.6n}
y(t, \Phi(t;x))=y_0(x),
\end{equation}
where $y_0:\R^3\to\R^3$ is any conveniently chosen diffeomorphic transformation. We choose
\begin{equation}\label{e1.7n}
y_0(x):=(x_1,x_2,\int_0^{x_3} \rhore(0,x_1,x_2,\s)\,d\s).
\end{equation}
From the relations \eqref{e1.5n}, \eqref{e1.6n} and \eqref{e1.7n} it follows that 
$$
J_y(t):=\det\left(\dfrac{\po y}{\po x}(t,\Phi(t;x))\right)
$$
satisfies
\begin{equation}\label{e1.8n}
\begin{cases}
\dfrac{dJ_y(t)}{dt}=-\div \ure(t,\Phi(t;x))\,J_y(t),\\
J_y(0)=\rhore(0,x).
\end{cases}
\end{equation}
Hence, again as in the non-relativistic case, we deduce
\begin{equation}\label{e1.9n}
\begin{aligned}
&\frac{d}{dt}\frac{\rhore(t,\Phi(t;x))}{J_y(t)} \\
&\qquad =\frac{\po_t\rhore(t,\Phi(t;x))+\ure(t,\Phi(t;x))\cdot\nabla\rhore(t,\Phi(t; x))J_y(t) 
 -J_y'(t)\rhore(t,\Phi(t;x))}{J_y^2(t)}\\
 &\qquad=\frac{-\div \ure(t,\Phi(t;x))\rhore(t, \Phi(t;x))J_y(t)+\div\ure(t,\Phi(t;x))J_y(t)\rhore(t,\Phi(t;x))}{J_y^2(t)}\\
 &\qquad=0.
 \end{aligned}
 \end{equation}
 Since $J_y(0)=\rhore(0,x)$, we conclude that 
 \begin{multline}\label{e1.10n}
 \det\left(\frac{\po y}{\po z}(t,\Phi(t;x))\right)=J_y(t)=\rhore(t,\Phi(t,x)), \quad\text{and so,}\\
 \det\left(\frac{\po y}{\po z}(t,z)\right)=\rhore(t,z)\quad \text{for all $(t,z)\in[0,\infty)\X\R^3$.}
 \end{multline}
 In particular, we see that as long as the relativistic density $\rhore(t,x)$ is positive, the Lagrangian transformation is non-singular.

\section{Main Property of the Massless Thirring Model}\label{S:3}
Concerning the $1+3$-dimensional case, short waves are described by an equation of the form 
\begin{equation}\label{e1.105}
\uu_t-\fka_1 \uu_{y_1} -\fka_2 \uu_{y_2}-\fka_3 \uu_{y_3}=-i {\mathfrak B}(t,y) \uu,
\end{equation}
where $\uu=\uu(t,y) \in\C^4$, $y\in\R^3$, $\fka_i$, $i=1,2,3$, are $4\X4$ complex matrices satisfying 
$\fka_i^*=\fka_i$, $\fka_i^2=I$, $\fka_i\fka_j=-\fka_j\fka_i$, $i\ne j$, $i,j=1,2,3$,
and  ${\mathfrak B}(t,y)$ is a $4\X4$ complex matrix such that, if ${\mathfrak B}^*$ denotes its adjoint matrix,  
${\mathfrak B}^*={\mathfrak B}$,  and $\fka_i {\mathfrak B}={\mathfrak B}\fka_i$, $i=1,2,3$. In the $1+3$-dimensional extension of the Thirring 
massless model 
\begin{equation}\label{e1.105'}
{\mathfrak B}(t,y)=\l {\mathfrak U}+V(t,y)I_4
\end{equation} 
with $\l$ a real constant, ${\mathfrak U}$ is the $4\X4$ quadratic Thirring matrix defined as 
$$
{\mathfrak U}:=\uu^\dag\uu I_4-\uu^\dag\fkb\uu\fkb
$$  
where
$$
\fkb:=i\fka_1\fka_2\fka_3,
$$
and $V$ is a real-valued function representing the potential of an external force.  The symbol ${}^\dag$ for column vectors in $\C^4$ means the conjugate transpose, i.e., if 
 $$
 \uu=\left(\begin{matrix}u_1\\ u_2\\ u_3\\u_4\end{matrix}\right)\in\C^4,
 $$ 
 then $\uu^\dag=(\bar u_1,\bar u_2,\bar u_3, \bar u_4)$, in particular, $\uu^\dag\uu=|\uu|^2=\sum_{i=1}^4\left((\Re u_i)^2+(\Im u_i)^2\right)$.
 
 We recall now the following result established in \cite{DF1}, which extended the previous analogue in \cite{DFi}, and constitutes the main property of the massless Thirring model. We reproduce its proof here for the convenience of the reader. 

\begin{theorem}\label{T:1.101} Let $\uu$ be a smooth solution of \eqref{e1.105} and let
$\fka_i$, $i=1,2,3$, satisfy the above properties. Then, both  $w=|\uu|^2$ and $w=\uu^\dag\fkb\uu$ satisfy
\begin{equation}\label{e1.89}
w_{tt}-w_{y_1y_1}-w_{y_2y_2}-w_{y_3y_3}=0.
\end{equation}
\end{theorem} 

\begin{proof} Multiplying \eqref{e1.105} by $\uu^\dag $ to the left
\begin{equation*}
\uu^\dag \uu_t-\uu^\dag \fka_1 \uu_{y_1}-\uu^\dag \fka_2 \uu_{y_2}-\uu^\dag\fka_3 \uu_{y_3}=-i\uu^\dag \fkB(t,y)\uu,
\end{equation*}
 applying ${}^\dag$ to the last equation
 $$
   \uu_t^\dag \uu - \uu_{y_1}^\dag\fka_1 \uu- \uu_{y_2}^\dag\fka_2 \uu - \uu_{y_3}^\dag\fka_3 \uu=i\uu^\dag \fkB(t,y) \uu,
 $$
 adding the last two gives
 \begin{equation}\label{e1.106'}
 (|\uu|^2)_t-(\uu^\dag\fka_1\uu)_{y_1}-(\uu^\dag\fka_2 \uu)_{y_2}-(\uu^\dag\fka_3 \uu)_{y_3}=0.
\end{equation}
Deriving the last equation by $t$, it follows
\begin{equation}\label{e1.106}
(|\uu|^2)_{tt}-(\uu^\dag\fka_1 \uu)_{ty_1}-(\uu^\dag\fka_2 \uu)_{ty_2}-(\uu^\dag\fka_3 \uu)_{ty_3}=0.
\end{equation}
Similarly, multiplying \eqref{e1.105}  by $u^\dag \fka_1$,  it follows,
\begin{equation*}
\uu^\dag\fka_1 \uu_t-\uu^\dag \uu_{y_1}-\uu^\dag \fka_1\fka_2 \uu_{y_2}-\uu^\dag\fka_1\fka_3 \uu_{y_3}=
-iu^\dag \fkB(t,y)\fka_1 u,
\end{equation*}  
applying ${}^\dag$ to the last equation
\begin{equation*}
 \uu_t^\dag \fka_1 \uu-\uu_{y_1}^\dag \uu- \uu_{y_2}^\dag\fka_2\fka_1  \uu-\uu_{y_3}^\dag\fka_3\fka_1  \uu=i\uu^\dag \fkB(t,y)\fka_1\uu,
\end{equation*}  
adding the last two gives
\begin{equation}\label{e1.107}
(\uu^\dag \fka_1 \uu)_t-(|\uu|^2)_{y_1}=0.
\end{equation}
 Similarly, we get
 \begin{equation}\label{e1.108}
(\uu^\dag \fka_2 \uu)_t-(|\uu|^2)_{y_2}=0,
 \end{equation}
and
\begin{equation}\label{e1.109}
(\uu^\dag \fka_3 \uu)_t-(|\uu|^2)_{y_3}=0.
\end{equation}
Deriving \eqref{e1.107} by $y_1$, \eqref{e1.108} by $y_2$ and \eqref{e1.109} by $y_3$ there follow, respectively, 
\begin{equation}\label{e1.110}
(\uu^\dag \fka_1 \uu)_{y_1t}-(|\uu|^2)_{y_1y_1}=0,
\end{equation}
\begin{equation}\label{e1.111}
(\uu^\dag \fka_2 \uu)_{y_2t}-(|\uu|^2)_{y_2y_2}=0,
\end{equation}
\begin{equation}\label{e1.112}
(\uu^\dag \fka_3 \uu)_{y_3t}-(|\uu|^2)_{y_3y_3}=0.
\end{equation}
Adding \eqref{e1.106}, \eqref{e1.110}, \eqref{e1.111} and \eqref{e1.112}, it follows
\begin{equation}\label{e1.113}
( |\uu|^2)_{tt}-(|\uu|^2)_{y_1y_1}-(|\uu|^2)_{y_2y_2}-(|\uu|^2)_{y_3y_3}=0,
\end{equation}
which proves the assertion for $w=|\uu|^2$. 

To prove the assertion for $w=\uu^\dag\fkb\uu$, we first multiply \eqref{e1.105} by $\uu^\dag\fkb$ on the left to obtain\footnote{Unfortunately, in \cite{DF1}, from the next equation up to the final one of the proof of this assertion, the signs of the terms with spatial derivatives appear as $+$ instead of $-$, due to a misprint.}
\begin{equation}\label{e1.114}
\uu^\dag\fkb\uu_t-\uu^\dag\fkb\fka_1\uu_{y_1}-\uu^\dag\fkb\fka_2\uu_{y_2}-\uu^\dag\fkb\fka_3\uu_{y_3}=-i\uu^\dag \fkB\fkb \uu.
\end{equation}
 We then apply ${}^\dag$ to \eqref{e1.114} and add the resulting equation to \eqref{e1.114} to obtain
\begin{equation}\label{e1.115}
(\uu^\dag\fkb\uu)_t-(\uu^\dag\fkb\fka_1\uu)_{y_1}-(\uu^\dag\fkb\fka_2\uu)_{y_2}-(\uu^\dag\fkb\fka_3\uu)_{y_3}=0.
\end{equation}
Deriving \eqref{e1.115} by $t$ we obtain
\begin{equation}\label{e1.115'}
(\uu^\dag\fkb\uu)_{tt}-(\uu^\dag\fkb\fka_1\uu)_{ty_1}-(\uu^\dag\fkb\fka_2\uu)_{ty_2}-(\uu^\dag\fkb\fka_3\uu)_{ty_3}=0.
\end{equation}
Now we multiply \eqref{e1.105} by $\uu^\dag\fkb\fka_1$ to get
\begin{equation}\label{e1.116}
\uu^\dag\fkb\fka_1\uu_t-\uu^\dag\fkb\uu_{y_1}-\uu^\dag\fkb\fka_1\fka_2\uu_{y_2}-\uu^\dag\fkb\fka_1\fka_3\uu_{y_3}=-i\uu^\dag \fkB\fkb\fka_1 \uu.
\end{equation}
We then apply ${}^\dag$ to \eqref{e1.116} and add the resulting equation to \eqref{e1.116}
to obtain
 \begin{equation*}
(\uu^\dag\fkb\fka_1\uu)_t-(\uu^\dag\fkb\uu)_{y_1}=0,
\end{equation*} 
which deriving with respect to $y_1$ gives
  \begin{equation}\label{e1.117}
(\uu^\dag\fkb\fka_1\uu)_{ty_1}-(\uu^\dag\fkb\uu)_{y_1y_1}=0.
\end{equation} 
Similarly, we obtain 
   \begin{equation}\label{e1.118}
(\uu^\dag\fkb\fka_2\uu)_{ty_2}-(\uu^\dag\fkb\uu)_{y_2y_2}=0,
\end{equation} 
and
  \begin{equation}\label{e1.119}
(\uu^\dag\fkb\fka_3\uu)_{ty_3}-(\uu^\dag\fkb\uu)_{y_3y_3}=0.
\end{equation} 
Adding \eqref{e1.115'}, \eqref{e1.117}, \eqref{e1.118} and \eqref{e1.119}  we then obtain \eqref{e1.89} for $w=\uu^\dag\fkb\uu$, which concludes the proof. 
\end{proof}

\begin{remark}\label{R:1} Observe that from  \eqref{e1.106'} at $t=0$ we obtain
\begin{equation}\label{e1.106''}
 (|\uu|^2)_t|_{t=0}=\left(\uu^\dag\fka_1\uu)_{y_1}+(\uu^\dag\fka_2 \uu)_{y_2}+(\uu^\dag\fka_3 \uu)_{y_3}\right)|_{t=0}.
\end{equation}
Similarly, from \eqref{e1.115} at $t=0$ we obtain
\begin{equation}\label{e1.115''}
(\uu^\dag\fkb\uu)_t|_{t=0}=\left((\uu^\dag\fkb\fka_1\uu)_{y_1}+(\uu^\dag\fkb\fka_2\uu)_{y_2}+(\uu^\dag\fkb\fka_3\uu)_{y_3}\right)|_{t=0}.
\end{equation}
The right-hand sides of both \eqref{e1.106''} and \eqref{e1.115''} are known from $\uu(0,x)$.
\end{remark}

\section{The model for the SW-LW interaction with relativistic fluid flows}\label{S:4}

We propose the following simplified model for the interaction of short waves described by the massless Thirring model \eqref{e1.105} and long waves governed by
the relativistic Euler equations \eqref{e1.3n}, with
\begin{align}
V(t,y)&:= \k  \vre(t,y), \quad \vre(t,y):= \frac1{\rhore\circ Y^{-1}(t,y)}, \label{e1.120}\\
F(t,x) &:=  \frac\a{\rhore (t,x)}\nabla_x\left(|\uu\circ Y(t,x)|^2\right), \label{e1.121}
\end{align}
where $Y^{-1}(t,y)=(t,x(t,y))$ is the inverse mapping of the Lagrangian transformation, $\k, \a$ are given positive constants. $\vre(t,y)$ is called the relativistic specific volume.

We then arrive at the following simplified model for the relativistic 
short wave-long wave interaction for the relativistic Euler equations
\begin{equation}\label{e1.122}
\begin{aligned}
& \po_t\uu-\fka_1 \po_{y_1}\uu -\fka_2 \po_{y_2}\uu-\fka_3 \po_{y_3}\uu=
-i \left(\l \fkU(t,y)+\k \vre(t,y)I_4\right)\uu,\\
&\po_t\rhore+\nabla_x\cdot (\rhore\ure)=0,\\
&\po_t(\rhore\ure)+\nabla_x\cdot(\rhore \ure\otimes\ure+ \tilde P+pI_3)=
 \a \nabla_x\left(|\uu\circ Y(t,x)|^2\right),
\end{aligned}
\end{equation}
where $(t,y)$ in the first equation are the Lagrangian coordinates of the fluid defined in \eqref{e1.6n0} and \eqref{e1.6n}. It is also useful to have the above system written in the original dependent variables:
\begin{equation}\label{e1.123}
\begin{aligned}
& \po_t\uu-\fka_1 \po_{y_1}\uu -\fka_2 \po_{y_2}\uu-\fka_3 \po_{y_3}\uu=
-i \left(\l \fkU(t,y)+\k \vre(t,y)I_4\right)\uu,\\
&\po_t\left(\frac{\rho+\ve^2 p}{1-\ve^2|u|^2}-\ve^2 p\right)+\sum_{k=1}^3\po_{x_k}\left(\frac{\rho+\ve^2p}{1-\ve^2|u|^2}u_k\right)=0,\\
&\po_t\left(\frac{\rho+\ve^2p}{1-\ve^2|u|^2}u_j\right)+\sum_{k=1}^3\po_{x_k}\left(\frac{\rho+\ve^2p}{1-\ve^2|u|^2}u_j u_k +p\d_{jk}\right)= \a \po_{x_j}\left(|\uu\circ Y(t,x)|^2\right) ,\\
&\qquad\qquad\qquad\qquad\qquad\qquad\qquad\qquad\qquad\qquad\qquad\qquad j=1,2,3.
\end{aligned}
\end{equation}

 \section{Short time smooth solutions for a regularized model}\label{S:5}
 
 The short time existence of smooth solutions to the Cauchy problem for the relativistic Euler equations  in several space variables \eqref{e1.1n} was addressed in a number of papers starting with \cite{MU1, MU2} and then \cite{PS,LFU}, among others. We will consider particularly here the article \cite{MU2}. The latter establishes a symmetrization of \eqref{e1.1n} based on the determination of a strictly convex entropy
 $\eta$ for \eqref{e1.1n} through an application of a theorem by Godunov in \cite{Go}.  The referred theorem in \cite{Go} establishes that, for a hyperbolic system in the dependent variable $w\in\R^m$, endowed with a strictly convex entropy $\eta$, the change of dependent variables $w\mapsto v=\nabla_w\eta(w)$, transforms the system in a hyperbolic symmetric system to which the local existence theory of a local in time smooth solution in \cite{K} applies.   
 
 Unfortunately, the system \eqref{e1.123} as it is has a serious imbalance. Indeed, from the analysis of the deformation gradient in, e.g.,  \cite{FPZ} (see also \cite{FMN})  we know that  the regularity of the gradient on right-hand side is compatible with that of  
 $\nabla_x\ure$. Since the equation itself only involves the derivatives of $\ure$ up to the first order this leaves us with a very challenging problem. This is different from the situation in, for instance, \cite{FPZ} and \cite{FMN}, where the long waves are described by the Navier-Stokes equations which is a parabolic system, not a first order hyperbolic system as in the present situation.  The situation here is also different from the one in \cite{DF}, where  one-dimensional models are considered that are implicitly or explicitly written in Lagrangian coordinates, in which case we can take the write-hand side as part of the flux function of the hyperbolic conservation law. 

As a way to repare the regularity imbalance commented above, here as a provisory solution to this problem, we consider the following regularization of \eqref{e1.123}:
\begin{equation}\label{e1.123r}
\begin{aligned}
& \po_t\uu-\fka_1 \po_{y_1}\uu -\fka_2 \po_{y_2}\uu-\fka_3 \po_{y_3}\uu=
-i \left(\l \fkU(t,y)+\k \vre(t,y)I_4\right)\uu,\\
&\po_t\left(\frac{\rho+\ve^2 p}{1-\ve^2|u|^2}-\ve^2 p\right)+\sum_{k=1}^3\po_{x_k}\left(\frac{\rho+\ve^2p}{1-\ve^2|u|^2}u_k\right)=0,\\
&\po_t\left(\frac{\rho+\ve^2p}{1-\ve^2|u|^2}u_j\right)+\sum_{k=1}^3\po_{x_k}\left(\frac{\rho+\ve^2p}{1-\ve^2|u|^2}u_j u_k +p\d_{jk}\right)= \a \po_{x_j}\left(\z_\d*|\uu\circ Y(t,x)|^2\right) ,\\
&\qquad\qquad\qquad\qquad\qquad\qquad\qquad\qquad\qquad\qquad\qquad\qquad j=1,2,3,
\end{aligned}
\end{equation}
 where we regularize the right-hand side of \eqref{e1.123} with a mollifier   
 $\z_\d(x)=\d^{-3}\z(\d^{-1}x)$, for some $\z\in C_c^\infty(\R^3)$, $\z\ge0$, and $\int_{\R^3}\z(x)\,dx=1$, for $\d>0$. 
 
  Let us pose the following initial conditions for \eqref{e1.123r}:
\begin{equation}\label{e1.124}
\begin{cases}
\uu|_{t=0}=\uu_0(x),\\
\rho|_{t=0}=\rho_0(x),\\
u_i|_{t=0}=u_{0i}(x), \quad i=1,2,3.
\end{cases}
\end{equation}

We adopt assumptions for the relativistic Euler equations as in \cite{MU2}. Let $0\le \rho_*<\rho^*\le \infty$. As in \cite{MU2}, we assume
\begin{equation}\label{e1.125}
\begin{aligned}
& p(\rho)\in C^\infty (\rho_*,\rho^*),\\
& p(\rho)>0,\quad 0<p'(\rho)<c^2 \quad \text{for $\rho\in (\rho_*,\rho^*)$}. 
\end{aligned}
\end{equation}

We have the following local in time existence and uniqueness of smooth solution for \eqref{e1.123r}-\eqref{e1.124}. Here, for simplicity, we avoid the use of uniformly local Sobolev spaces as in \cite{K} and \cite{MU2}.

 \begin{theorem} \label{T:main} Assume \eqref{e1.125} for $p$. Suppose the initial data $\uu_0$, 
  and $(u_{01}, u_{02}, u_{03})$ belong to $H^s(\R^3)$ and $\rho_0-\rho_\infty\in H^s(\R^3)$, for some $\rho_\infty\in(\rho_*,\rho^*)$, $s>\frac52$, and that there exist a positive constant $\d$ sufficiently small so that 
 \begin{gather}
 \rho_*+\d\le \rho_0(x)\le \rho^*-\d,\\
 |u_0|^2(x)=u_{01}^2(x)+u_{02}^2(x)+u_{03}^2(x)\le (1-\d)c^2,
 \end{gather}
 for all $x\in\R^3$. Then, the Cauchy problem \eqref{e1.123r}-\eqref{e1.124} has a unique solution
 \begin{equation}\label{e1.126}
 (\uu,\rho-\rho_\infty,u)\in L^\infty(0,T; H^s)\cap C([0,T];H^{s-1}),
 \end{equation}
 with $\rho_*<\rho(t,x)<\rho^*$ and $|u|^2(t,x)<c^2$. Here $T>0$ depends only on $\d$ and the $H^s$-norm of $(\uu_0,\rho_0-\rho_\infty,u_0)$.
 \end{theorem}
 
 \begin{proof} The proof proceeds through the following arguments. We use the symmetrization found in \cite{MU2} in order to transform the 
 left-hand side of \eqref{e1.123r}  into the form 
 \begin{equation}\label{e1.127}
 \LL(V)=a_0(V)\po_t V(t)+\sum_{j=1}^3 a_j(V)\po_{x_j}V(t),
 \end{equation}
 where $a_0$, $a_j$, $j=1,2,3$,  are $4\X4$ symmetric matrices, with $a_0$ positive definite. The right-hand side of \eqref{e1.123r} can be seen as a functional  
 ${\mathcal R}(V)$ on $C([0,T];H^s(\R^3;\R^4))$ into itself in the following way.  
 Given
 $V\in C([0,T];H^s(\R^3;\R^4))$, we may obtain the corresponding $(\rho^V, u^V)$ and then the corresponding $\ure^V$, from which we can define the Lagrangian transformation $Y^V$ as in Section~\ref{S:2}. Then, we obtain $|\uu|^2$ in the coordinates $(t,y)$ using Theorem~\ref{T:1.101} and the Remark~\ref{R:1}, solving the wave equation with initial data $|\uu_0|^2$ and \eqref{e1.115''}. This way we obtain $|\uu\circ Y^V|^2$, and so $\z_\d*|\uu\circ Y^V|^2$, and we define
 $$
 {\mathcal R}[V](t)=\left(\begin{matrix} 0\\  \nabla_x\z_\d*|\uu\circ Y^V(t)|^2\end{matrix}\right).
 $$
 By the regularity properties of the Lagrangian transformation obtained, e.g., in \cite{FPZ,FMN}, $\cR$ maps $C([0,T];H^s(\R^3;\R^4))$ into itself. 
   So we arrive at the $4\X4$ symmetric system
  \begin{equation}\label{e1.128}
  a_0(V)\po_t V(t)+\sum_{j=1}^3 a_j(V)\po_{x_j}V(t)={\mathcal R}[V](t).  
 \end{equation}
 This system does not entirely fall into the form of (Q') in \cite{K}, since the right-hand side is not of the form $F(t)[u(t)]$ as the right-hand side of (Q') in \cite{K}, which is defined, for each $t\in[0,T]$, as a function on a subset $D$ of $H^s(\R^3; P)$  where $P=\R^4$ here. Nevertheless, the fixed point argument in \cite{K} may be easily adapted to the present situation by the  following reason. As in \cite{K}, let us consider the map $\Phi$ defined therein in p.198, over which the fixed point argument is carried out. Following the same arguments in \cite{K}, we can prove, by arguments similar to those employed in Lemma~4.4 of \cite{K},  that  $\Phi$ maps $S$ into itself, if $L'$ and $T'$ are appropriately chosen, where $S$ is the set of functions from $[0,T']$ to $H^s(\R^3;\R^4)$, Lipschitz continuous in time in the $H^{s-1}$-norm, with Lipschitz constant $L'$,   defined in p.195 of \cite{K}.  Having that at hand we may conclude the fixed point argument as in \cite{K}. The key difference here, as compared with \cite{K}, lies in the proof of the estimate $\|f^v(t)-f^w(t)\|\le\mu\,d(v,w)$ in Lemma~4.5, p.198 of \cite{K},  which here would translate into
 \begin{equation}\label{e1.129}
  \|\cR[V](t)-\cR[W](t)\|_{L^2(\R^3)}\le C \sup_{t\in[0,T]}\|V(t)- W(t)\|_{L^2(\R^3)},
 \end{equation}
 for some constant $C>0$. To prove \eqref{e1.129} we first note that, by Young's inequality for convolutions, the left-hand side is bounded by
 $$
 C\||\uu(t,y^V(t,\cdot))|^2-|\uu(t,y^W(t,\cdot)|^2\|_{L^2(\R^3)}
 $$
 for some constant $C>0$ which may vary along this proof. Now, since $|\uu(t,y)|^2$ is the solution of the wave equation with smooth initial data by Theorem~\ref{T:1.101}, and Remark~\ref{R:1}, the above quantity is bounded by
 $$
 C\|y^V(t,\cdot)-y^W(t,\cdot)\|_{L^2(\R^3)}.
 $$
 Now, we argue as in \cite{FMN}, p.148. We have, from the definition of the Lagrangian coordinate,
 \begin{align*}
 &\po_ty^V(t,x)+\ure^V(t,x)\cdot\nabla_xy^V(t,x)=0,\\
 &y^V(0,x):=(x_1,x_2,\int_0^{x_3} \rhore(0,x_1,x_2,\s)\,d\s)
\end{align*}
 and a similar equation holds for $y^W(t,x)$ with the same initial condition, which does not depend on $V$ or $W$. Therefore, denoting, as in \cite{FMN}, $\tilde y(t,x)=y^V(t,x)-y^W(t,x)$, we obtain that $\tilde y(t,x)$ satifies
 $$
 \begin{cases} 
 \tilde y_t=-(\ure^V-\ure^W)\cdot\nabla_x y^V-\ure^W\cdot\nabla_x\tilde y,\\
 \tilde y(0,x)=0.
 \end{cases}
 $$
Multiplying by $\tilde y$ and integrating by parts we obtain (as in \cite{FMN}) we obtain
$$
\frac{d}{dt}\int_{\R^3}|\tilde y|^2\,dx\le C(\|\nabla y^V\|_\infty^2\|\ure^V-\ure^W\|_{L^2(\R^3)}+
\|\div \ure^W\|_\infty^2\|\tilde y\|_{L^2(\R^3)}^2),
$$
and then, using the boundedness of $\|\nabla_x y^V\|_\infty$ (see, e.g., (2.30) in \cite{FMN}, where here $\ure$ plays the role of $u$ therein), and the time integrability of $\|\div \ure^W\|_\infty$, by Gronwall's inequality, we obtain 
$$
\|\tilde y\|_{L^2(\R^3)}^2\le C\sup_{t\in[0,T]}\|\ure^V(t)-\ure^W(t)\|_{L^2(\R^3)},
$$ 
 from which, using the fact that $(\rhore, \ure)\mapsto(\rho,u)$ and $(\rho,u)\mapsto V$ are local bi-Lipschitz diffeomorphisms, \eqref{e1.129} follows. Thus, we prove as in Lemma~4.5 of \cite{K} that the map which here corresponds to $\Phi$ therein is a contraction map  of $S$ endowed with the metric
 $$
 d(V,W)=\sup_{0\le t\le T'}\|V(t)-W(t)\|_{L^2(\R^3)}
 $$
 if $T'$ is sufficiently small.
  
 Once we have obtained the solution of \eqref{e1.128}, from $V(t)$ and the relativistic Lagrangian transformation $Y$ we get $\vre(t,y)$ and from Theorem~\ref{T:1.101} we also
 obtain $\fkU$ from the initial data. Therefore, the Dirac equation in \eqref{e1.123r} 
 reduces to a linear equation with a right-hand side of the form $-i\fkB(t,y)\uu$, with $\fkB(t,y)$ known, satisfying the hypotheses of Theorem~\ref{T:1.101}, whose solution is standard. 
 
 Uniqueness follows first from the uniqueness of the fixed point of contractions; second from the uniqueness of the smooth solution  of the Cauchy problem for a linear Dirac equation. This completes the proof.  
 
 \end{proof}
  
  \begin{remark}\label{R:final} It can be seen, by checking the arguments above and in \cite{K}, that when $\d\sim \a \to0$, the local in time solution given by Theorem~\ref{T:main} converges to the local in time smooth solution of the relativistic Euler equation obtained in \cite{MU2}, together with a smooth solution of the $1+3$-dimensional Thirring model nonlinear Dirac equation with a potential proportional to the relativistic specific volume. It is very interesting the fact that in this limit system short waves and long waves are still coupled, not only through the Lagrangian coordinates, but also  through the relativistic specific volume in the Dirac equation!  
  \end{remark}


\begin{thebibliography}{999}

\bibitem{Ba} A.~Bachelot. {\sl Global existence of large amplitude solutions for nonlinear massless Dirac equation}. Portugaliae Mathematica, {\bf46} (1989), 
455--473.

\bibitem{Bat} G. K. Batchelor. ``An Introduction to Fluid Dynamics.'' Cambridge University Press, 1967.

\bibitem{BOP} D.~Bekiranov, T.~Ogawa, G.~Ponce. {\sl Weak solvability and well-posedness of a coupled Schr\u odinger--Korteweg De Vries equation for capillary-gravity wave interactions}. Proc.\ Am.\ Math.\ Soc.\ {\bf 125}(10) (1997), 2907--2919. 

\bibitem{Be} D.J.~Benney. {\sl A general theory for interactions between short and long waves}. Studies in Applied Mathematics {\bf 56} (1977), 81--94. 	

\bibitem{Br} Y.~Brenier. {\sl Hydrodynamic structure of the augmented Born-Infeld equations}. 
Arch.\ Rational Mech.\ Anal.\ {\bf172} (2004), 65--91.

\bibitem{Ca} T.~Candy. {\sl Global existence for an $L^2$ critical nonlinear Dirac equation in one dimension}.
Adv.~Differential Equations {\bf16} (2011), no. 7-8, 643--666.

\bibitem{CH} T.~Candy, S.~Herr. {\sl On the Majorana condition for nonlinear Dirac systems}. Ann.~I.~H.~Poincar\'e -- AN {\bf35} (2018) 1707--1717.
 

\bibitem{CG} J.M.~Chadam, R.T.~Glassey. {\sl On certain global solutions for the (classical) coupled 
Klein-Gordon-Dirac equations in one and three space dimensions.} Arch.\ Rational Mach.\ Anal.\ {\bf 54} (1974),
223--237. 

\bibitem{Da} C.~Dafermos. ``Hyperbolic Conservation Laws in Continuum Physics''. Third Edition.
Springer-Verlag, 2010.  

\bibitem{Del} V.~Delgado. {\sl Global solutions of the Cauchy problem for the (classical) coupled Maxwell-Dirac and other nonlinear Dirac equations in one space dimension}. Proceedings of the American Mathematical Society {\bf69}, No.~2, 289--296. 


\bibitem{DFi} J.-P.\ Dias, M.~Figueira. {\sl Time decay for the solutions of a nonlinear Dirac equation
in one space dimension}.  Ricerche di Matematica, {\bf 35}, No.2 (1986), 309--316.

\bibitem{DFO}  J.-P.\ Dias , M.~Figueira, F.~Oliveira. {\sl Existence of local strong solutions for a quasilinear Benney system.}  C.\ R.\ Acad.\ Sci.\ Paris, Ser.\ 1, {\bf 344} (2007), 493--496. 


%\bibitem{DF2} J.-P.\ Dias, M.~Figueira. {\sl Existence of weak solutions for a quasilinear version of Benney equations.} J.\ Hyperbolic Differ.\ Equ.\ {\bf4} (2007), no. 3, 555--563. 

%\bibitem{DFF}  J.-P.~Dias, M.~Figueira and H.~Frid. {\sl Vanishing viscosity with short wave long wave interactions for systems of conservation laws}.   Arch.\ Ration.\ Mech.\ Anal.\  {\bf196}  (2010),  no. 3, 981--1010. 
	
\bibitem{DF} J.-P.~Dias and H.~Frid. {\sl Short wave-long wave interactions for compressible Navier-Stokes equations}. SIAM J.\ Math.\ Anal., {\bf43} (2011), 764--787.

\bibitem{DF1} J.-P.~Dias and H.~Frid. {\sl Short wave-long wave interactions in the relativistic context}. Comm.\ Math.\ Anal.\ Appl.\ Vol.~{\bf 1}, No.~2,  (2022), 263--284.

\bibitem{Di} P.A.M.~Dirac. ``The Principles of Quantum Mechanics''. Fourth Edition (Revised). Clarendon Press-Oxford. Oxford University Press, 1958. 

\bibitem{FJP}  H.~Frid, J.~Jia, R.~Pan. {\sl Global smooth solutions in R3 to short wave–long wave interactions in magnetohydrodynamics}, J. Differential Equations 262 (2017), no. 7, 4129–4173.	

\bibitem{FMN} H.~Frid, D.R.~Marroquin, J.F.C.~ Nariyoshi. {\sl Global smooth solutions with large data for a system modeling aurora type phenomena in the 2-torus}. SIAM J.\ Math.\ Anal.\ {\bf53} (2021), no.\ 1, 1122--1167.	
	
\bibitem{FMP} H. Frid, D.R.~Marroquin, R.~Pan. {\sl Modeling aurora type phenomena by short wave-long wave interactions in multidimensional large magnetohydrodynamic flows}. SIAM J.\ Math.\ Anal., 50(6) (2018), 6156--6195.
	
\bibitem{FPZ} H.~Frid, R.~Pan and W.~Zhang. {\sl Global smooth solutions in R3 to short wave-long wave interactions systems for viscous compressible fluids}, SIAM J. Math. Anal., Vol. 46, No. 3 (2014), pp. 1946--1968.	

\bibitem{Go} S.K.~Godunov. {\sl An interesting class of quasi-linear systems.} Dokl.\ Acad.\ Nauk SSSR, {\bf 139} (1961), 521--523. 

\bibitem{Ho} L.~H\"ormander. ``Lectures on Nonlinear Hyperbolic Differential Equations''. Springer-Verlag, 1996.

\bibitem{HW} S.~Huo, C.~Wei. {\sl Classical solutions to relativistic Burgers equations in FLRW space-times.} 
Science China Mathematics {\bf 63}, No.~2 (2020), 357--370. 

\bibitem{K} T.~Kato. {\sl The Cauchy problem for quasi-linear symmetric hyperbolic systems.} Arch.\ Rational Mech.\ Anal.\ {\bf 58} (1975), 181--205.  

\bibitem{KM} E.A.~Kuznetsov, A.V.~Mikhailov. {\sl On the complete integrability of the two-dimensional classical Thirring model}. Theoretical and Mathematical Physics {\bf30} (1977), 193--200.

\bibitem{LFMO} Ph.~Lefloch, H.~Makhlof, B.~Okutmustur. {\sl Relativistic Burgers equations on curved spacetimes. 
Derivation and finite volume approximation.} SIAM J.\ Numer.\ Anal.\ {\bf50}, No.~4 (2012), 2136--2158.  

\bibitem{LFU} Ph.~LeFloch, S.~Ukai. {\sl A symmetrization of the Euler relativistic equations with several spatial variables}. Kinetic and Related Models, Vol.~{\bf 2}, No.~2 (2009), 275--292.   

\bibitem{MO} S.~Machihara, T.~Omoso. {\sl The explicit solutions to the nonlinear Dirac equation and 
Dirac-Klein-Gordon equation}. Ricerche di Matematica {\bf 56} (2007), 19--30. 

\bibitem{MU1} T.~Makino, S.~Ukai. {\sl Local smooth solutions of the relativistic Euler 
equation.} J.\ Math.\ Kyoto Univ., {\bf 35}-1 (1995), 105--114. 

\bibitem{MU2} T.~Makino, S.~Ukai. {\sl Local smooth solutions to the relativistic Euler equation, II}. Kodai Math.\ J.\ {\bf 18} (1995), 355--375.    

\bibitem{DRM} D.R.~Marroquin. {\sl Vanishing viscosity limit of short wave-long wave interactions in planar magnetohydrodynamics}, J.\ Differential Equations {\bf 266} (2019), no.~12, 8110--8163.

%\bibitem{Mu} F.~Murat. {\sl L'injection du c\^one positif de $H^{-1}$ dans $W^{-1,q}$ est compacte pour tout q<2.} (French) [The injection of the positive cone of $H^{-1}$ in $W^{-1,q}$ is completely continuous for all q<2] J.\ Math.\ Pures Appl.\ (9) {\bf60} (1981), no.~3, 309--322.

%\bibitem{NS} W.~Neves, D.~Serre {\sl The incompleteness of  the Born-Infeld model for nonlinear multi-d
%Maxwell's Equations}. Quart.\ Appl.\ Math.\ {\bf LXIII} (2005), 343--368.

\bibitem{PS} R.~Pan, J.A.~Smoller. {\sl Blowup of smooth solutions for relativistic Euler equations.} Comm.\ Math.\ Phys.\ {\bf 262} (2006), 729--755. 

\bibitem{Se} D.~Serre. ``Systems of conservation laws'', 2. Ch.~9.6 and 10.1, Cambridge University Press, 2000.






\bibitem{St} W.~Strauss. {\sl Nonlinear invariant wave equation}; in ``Invariant Wave Equations'', Proceedings of the "Etore Majorana" International School of Physics Held in  
in Erice, June 27 to July 9, 1977, 197--249.


%\bibitem{Ta} L.~Tartar. {\sl Compensated compactness and applications to partial differential equations}. In: Nonlinear Analysis and Mechanics: Heriot-Watt Symposium, vol. IV, pp.136--212. Res. Notes in Math., 39. Pitman, Boston, 1979

\bibitem{Th} W.~Thirring. {\sl A soluble relativistic field theory}. Annals of Physics, {\bf 3} (1958), 91--112.

\bibitem{TH} M.~Tsutsumi, S.~Hatano.{\sl Well-posedness of the Cauchy problem for the long wave-short wave resonance equations}. Nonlinear Anal.\ Theory Methods Appl.\ {\bf22} (2) (1994), 155--171.

\bibitem{ZZ} Y.~Zhang, Q.~Zhao. {\sl Global solution to nonlinear Dirac equation for
Gross-Neveu  model in 1+1 dimensions. } Nonlinear Analysis {\bf 118} (2015), 82--96. 



\end{thebibliography}
\end{document}